\documentclass{amsart}[12pt]

\usepackage{yfonts} %
\usepackage{amssymb} %
\usepackage{amsthm}
\usepackage{array}
\usepackage{booktabs}%
\usepackage{hhline}%
\usepackage{xy} %
\usepackage{epsfig}%
\usepackage{color}%
\usepackage{upgreek}
\usepackage[english]{babel}
\usepackage{epigraph}%
\usepackage{fancybox}%
\usepackage{shadow}
\usepackage{afterpage}
\usepackage{mathrsfs}
\usepackage{enumitem}
\usepackage{tabularx}
\usepackage{subcaption}
\usepackage{graphicx}
\usepackage{type1cm}
\usepackage{eso-pic}
\usepackage{color}
\usepackage{upgreek}
\usepackage[foot]{amsaddr}

\newtheorem{theorem}{Theorem}
\newtheorem{lemma}{Lemma}

\theoremstyle{definition}
\newtheorem{definition}{Definition}
\newtheorem{remark}{Remark}

\theoremstyle{plain}
\newtheorem{corollary}{Corollary}

\newcommand{\vt}{\vspace{.1cm}}

\newcommand{\R}{\mathbb{R} }
\newcommand{\q}{\mathbb{Q}_{\epsilon}^n }

\newcommand{\h}{\mathbb{H} }

\newcommand{\s}{\mathbb{S}}

\renewcommand{\rho}{\varrho}
\renewcommand{\theta}{\varTheta}
\renewcommand{\Theta}{\varTheta}
\renewcommand{\Sigma}{\varSigma}
\renewcommand{\tau}{\uptau}
\usepackage{amsmath}

\newcommand{\overbar}[1]{\mkern 1.5mu\overline{\mkern-1.5mu#1\mkern-1.5mu}\mkern 1.5mu}

\newcommand{\transv}{\mathrel{\text{\tpitchfork}}}
\makeatletter
\newcommand{\tpitchfork}{%
  \vbox{
    \baselineskip\z@skip
    \lineskip-.52ex
    \lineskiplimit\maxdimen
    \m@th
    \ialign{##\crcr\hidewidth\smash{$-$}\hidewidth\crcr$\pitchfork$\crcr}
  }%
}
\makeatother

\begin{document}

\title[Totally Umbilical Hypersurfaces]
{Totally Umbilical Hypersurfaces \\ of   Product Spaces}
\author{Ronaldo F. de Lima \and João Paulo dos Santos}
\address[A1]{Departamento de Matem\'atica - Universidade Federal do Rio Grande do Norte}
\email{ronaldo@ccet.ufrn.br}
\address[A2]{Departamento de Matem\'atica - Universidade de Brasília}
\email{joaopsantos@unb.br}
\subjclass[2020]{53B25 (primary), 53C24,  53C42 (secondary).}
\keywords{umbilical  --  product space -- warped product.}

\maketitle

\begin{abstract}
Given a Riemannian manifold $M,$   and an open interval $I\subset\mathbb{R},$
we characterize nontrivial  totally umbilical hypersurfaces of the product
$M\times I$ --- as well as of  warped products $I\times_\omega M$ --- as those which are local
graphs built on isoparametric families of totally umbilical hypersurfaces of $M.$
By  means of this characterization, we
fully extend to $\mathbb{S}^n\times\mathbb{R}$ and $\mathbb{H}^n\times\mathbb{R}$ the
results by Souam and Toubiana on the classification of totally umbilical hypersurfaces
of $\mathbb{S}^2\times\mathbb{R}$ and $\mathbb{H}^2\times\mathbb{R}.$ It is also shown
that an analogous classification holds for
arbitrary warped products $I\times_\omega\mathbb{S}^n$ and $I\times_\omega\mathbb{H}^n.$
\end{abstract}

\section{Introduction}
In submanifold theory, one of the primary concerns is the determination of the totally
umbilical hypersurfaces of a given ambient space $\overbar M.$ Here, we consider this problem when
$\overbar M$ is a product manifold $M\times I,$ and  more generally, a warped product $I\times_\omega M,$ where
$I\subset\R$ is an open interval, and $M$ is a Riemannian manifold.

Trivial examples of totally umbilical hypersurfaces of $M\times I$, which are in fact totally geodesic,
are the horizontal sections
$M\times\{t\},\, \, t\in I,$ and the vertical cylinders over totally geodesic hypersurfaces of
$M$ (if any).  Our first main result establishes that
nontrivial totally umbilical (resp. totally geodesic) hypersurfaces exist in $M\times I$ if and only
if $M$ admits local isoparametric families of totally umbilical (resp. totally geodesic) hypersurfaces. (Recall that a
family of hypersurfaces is called \emph{isoparametric}
if its  members are parallel  and have constant mean curvature.)

More precisely, we show that  a nontrivial totally umbilical (resp. totally geodesic) hypersurface of $M\times\R$
is necessarily a local graph over an open set of $M,$ whose level sets constitute an isoparametric
family of totally umbilical (resp. totally geodesic) hypersurfaces. Conversely, any such graph
is totally umbilical (resp. totally geodesic) in $M\times\R.$ This result  extends naturally to warped products
$I\times_\omega M$, since such a manifold is conformal to  $M\times F(I),$
where $F'=1/\omega,$ and umbilicity is preserved under conformal diffeomorphisms.
As a first consequence of  these results, we recover the characterization
of totally umbilical hypersurfaces of $M\times I$ and $I\times_\omega M$ given
by Souam and Van der Veken in \cite{souam-veken}.

In \cite{souam-toubiana}, Souam and Toubiana  classified the
totally umbilical hypersurfaces of $\s^2\times\R$ and $\h^2\times\R$. They
constructed nontrivial  complete  totally umbilical hypersurfaces
in theses spaces, all  invariant by  rotational or translational isometries, and proved that
they are properly embedded, analytic, and  unique up to ambient isometries.

By means of our characterization of totally umbilical hypersurfaces of product spaces,
we not only provide  new  proofs of Souam and Toubiana's results, but also fully extend them
to $\s^n\times\R$ and $\h^n\times\R.$ In addition, we show that an analogous
classification of the totally umbilical hypersurfaces of $I\times_\omega\s^n$
and $I\times_\omega\h^n$ holds for any warping function $\omega.$

It should be mentioned that, in \cite{mendonca-tojeiro}, a complete classification
of totally umbilical hypersurfaces of
arbitrary codimension in $\s^n\times\R$ was obtained. When the codimension is one, 
this classification coincides with ours.
Nevertheless, the methods we employ here
are different from the ones applied in that work, being our proofs somewhat simpler.

The paper is organized as follows. In Section \ref{sec-preliminaries}, we introduce some notation
and formulae. In Section \ref{sec-graphs}, we discuss on graphs  of $M\times\R$ (called $(f_s,\phi)$-graphs)
whose level hypersurfaces  are parallel  in $M.$
In Section \ref{sec-lemmas}, we establish three key lemmas,
and in Section \ref{sec-umbilicalMxI} we characterize totally umbilical hypersurfaces of
$M\times\R$ and $I\times_\omega M$ by means of $(f_s,\phi)$-graphs. Finally, in Sections
\ref{sec-umbilicalQnxR} and  \ref{sec-Umbilicalwarp}, we classify the totally umbilical hypersurfaces
of $\s^n\times\R$ and $\h^n\times\R$, as well as those of $I\times_\omega\s^n$ and $I\times_\omega\h^n.$

\section{Preliminaries} \label{sec-preliminaries}
Let $\Sigma^n$, $n\ge 2,$
be an oriented hypersurface of  a Riemannian manifold $\overbar M^{n+1}.$ Set
$\overbar\nabla$ for the Levi-Civita connection of $\overbar M,$
$N$ for the unit normal field of $\Sigma,$ and   $A$ for its shape operator  with respect to
$N,$ so that
\[
AX=-\overbar\nabla_XN,  \,\, X\in T\Sigma,
\]
where  $T\Sigma$ stand for  the tangent bundle of $\Sigma$.
The principal curvatures of $\Sigma,$ that is,
the eigenvalues of the shape operator  $A,$ will be denoted by $k_1\,, \dots ,k_n$.

\begin{definition}
A hypersurface $\Sigma$ of $\overbar M$ is called \emph{totally umbilical} if its
shape operator writes as $A=\lambda \,{\rm Id},$ where $\lambda$ is a differentiable
function on $\Sigma$ (called its \emph{umbilical function}) and ${\rm Id}$ is the identity
operator on $T\Sigma.$ A totally umbilical hypersurface with null umbilical function is called
\emph{totally geodesic}.
\end{definition}

The ambient spaces $\overbar M$ we shall consider are the  Riemannian products
$M\times I$ and the warped products $I\times_\omega M$, whose  metrics are
\[
\langle\,,\,\rangle=\langle\,,\,\rangle_{M}+dt^2 \quad\text{and}\quad
\langle\,,\,\rangle=dt^2+\omega^2\langle\,,\,\rangle_{M},
\]
respectively. Here,  $M$ is an arbitrary Riemannian manifold,
$I\subset\R$ is an open interval, and $\omega$ a differentiable positive function on $I.$

Let $\overbar M$ be either of these product spaces.
We denote by $\partial_t$ the  gradient  of the
projection $\pi_{\scriptscriptstyle I}$ of $\overbar M$  on the factor $I.$
Notice that, in contrast with the Riemannian product case, in the warped product case,
$\partial_t$ is not a parallel field.

Given a hypersurface  $\Sigma$  of $\overbar M,$ its
\emph{height function} $\xi$ and its \emph{angle function} $\Theta$
are defined by the following identities:
\[
\xi(x):=\pi_{\scriptscriptstyle I}|_\Sigma \quad\text{and}\quad \Theta(x):=\langle N(x),\partial_t\rangle, \,\, x\in\Sigma.
\]

We shall denote the gradient of $\xi$ on $\Sigma$ by $T,$ so that
\begin{equation} \label{eq-gradxi}
T=\partial_t-\Theta N.
\end{equation}

\begin{definition}
A point of $\Sigma$ at which $T$ vanishes will be  called \emph{horizontal}.
\end{definition}

Given $t\in\R,$ the set  $P_t:=M\times\{t\}$ is called a \emph{horizontal hyperplane}
of $M\times\R.$ Horizontal hyperplanes are all isometric to $M$ and totally geodesic in
$M\times\R.$ In this context,
we call a transversal intersection $\Sigma_t:=\Sigma\transv P_t$ a \emph{horizontal section}
of $\Sigma.$ Any horizontal section $\Sigma_t$ is a hypersurface of $P_t$\,.
So, at any point $x\in\Sigma_t\subset\Sigma,$ the tangent space $T_x\Sigma$ of $\Sigma$ at $x$ splits
as the orthogonal sum
\begin{equation}\label{eq-sum}
T_x\Sigma=T_x\Sigma_t\oplus {\rm Span}\{T\}.
\end{equation}

Analogously, we define a \emph{vertical hyperplane}
$P_t=\{t\}\times_\omega M$ in $I\times_{\omega} M$, and a \emph{vertical section}
$\Sigma_t:=\Sigma\transv P_t$ of a hypersurface $\Sigma$ of $I\times_\omega M.$ We point out that
the vertical hyperplanes $P_t$ are totally umbilical in $I\times_\omega M$ with constant umbilical
function $\omega'(t)/\omega(t)$ (see, e.g., \cite{bishop-oneill}).

Given a totally geodesic hypersurface $\Sigma_0$ of $M,$ it is easily checked that the products
$\Sigma_0\times I$ and $I\times_\omega\Sigma_0$ are totally geodesic in
$M\times I$ and $I\times_\omega M,$ respectively. These hypersurfaces, together with
the horizontal hyperplanes of $M\times\R$ and vertical hyperplanes of $I\times_\omega M,$
will be referred as the \emph{trivial umbilical hypersurfaces} of these product spaces.

We will adopt the unified notation $\q$
for the simply connected space form of constant sectional
curvature $\epsilon\in\{1,-1\},$  so that $\mathbb Q_1^n$ is
the unit sphere $\s^n$, and $\mathbb Q_{-1}^n$ is the hyperbolic space $\h^n$.

Finally, we recall that a hypersurface $\Sigma\subset \q\times I$ is called \emph{rotational} if it
is invariant by a one-parameter group of rotations of $\q\times I$ fixing an \emph{axis} $\{o\}\times I,$
$o\in \q.$ Notice that, such a $\Sigma$ is foliated by (parts of) horizontal
$(n-1)$-spheres with center on the axis, so that
they project on a family of parallel geodesic spheres of $M$ centered at $o.$

\section{Graphs on  Parallel Hypersurfaces}  \label{sec-graphs}
In this section, we introduce  graphs in $M\times I$ which
are built on families of parallel hypersurfaces of $M.$ Through them,
we shall construct and classify totally umbilical hypersurfaces of $M\times I.$

With this purpose, consider
an isometric immersion
\[f:M_0^{n-1}\rightarrow M^n\]
between two Riemannian manifolds $M_0^{n-1}$ and $M^n,$
and suppose that there is a neighborhood $\mathscr{U}$
of $M_0$ in $TM_0^\perp$  without focal points of $f,$ that is,
the restriction of the normal exponential map $\exp^\perp_{M_0}:TM_0^\perp\rightarrow M$ to
$\mathscr{U}$ is a diffeomorphism onto its image. In this case, denoting by
$\eta$ the unit normal field  of $f,$   there is an open interval $I_0\owns 0,$
such that, for all $p\in M_0,$ the curve
\begin{equation}\label{eq-geodesic}
\gamma_p(s)=\exp_{\scriptscriptstyle M}(f(p),s\eta(p)), \, s\in I_0,
\end{equation}
is a well defined geodesic of $M$ without conjugate points. Thus,
for all $s\in I_0,$
\[
\begin{array}{cccc}
f_s: & M_0 & \rightarrow & M\\
     &  p       & \mapsto     & \gamma_p(s)
\end{array}
\]
is an immersion of $M_0$ into $M,$ which is said to be \emph{parallel} to $f.$
Observe that, given $p\in M_0$, the tangent space $f_{s_*}(T_p M_0)$ of $f_s$ at $p$ is the parallel transport of $f_{*}(T_p M_0)$ along
$\gamma_p$ from $0$ to $s.$ We also remark that,  with the induced metric,
the unit normal  $\eta_s$  of $f_s$ at $p$ is given by
\[\eta_s(p)=\gamma_p'(s).\]

\begin{definition}
Let $\phi:I_0\rightarrow \phi(I_0)\subset\R$ be an increasing diffeomorphism, i.e., $\phi'>0.$
With the above notation, we call the set
\begin{equation}\label{eq-paralleldescription1}
\Sigma:=\{(f_s(p),\phi(s))\in M\times\R\,;\, p\in M_0, \, s\in I_0\},
\end{equation}
the \emph{graph} determined by $\{f_s\,;\, s\in I_0\}$ and $\phi,$ or $(f_s,\phi)$-\emph{graph}, for short.
\end{definition}

Notice that, for a given $(f_s,\phi)$-graph $\Sigma$, and
for any $s\in I_0$\,, $f_s(M_0)$ is  the projection on $M$ of the horizontal
section $\Sigma_{\phi(s)}\subset\Sigma,$
that is, these sets are the {level hypersurfaces}
of    $\Sigma.$

For an arbitrary  point $x=(f_s(p),\phi(s))$ of such a graph $\Sigma,$ one has
\[T_x\Sigma=f_{s_*}(T_p M_0)\oplus {\rm Span}\,\{\partial_s\}, \,\,\, \partial_s=\eta_s+\phi'(s)\partial_t.\]
So, a  unit normal  to $\Sigma$ is
\begin{equation} \label{eq-normal}
N=\frac{-\phi'}{\sqrt{1+(\phi')^2}}\eta_s+\frac{1}{\sqrt{1+(\phi')^2}}\partial_t\,.
\end{equation}
In particular, its  angle function  is
\begin{equation} \label{eq-thetaparallel}
\Theta=\frac{1}{\sqrt{1+(\phi')^2}}\,\cdot
\end{equation}

As shown in  \cite[Theorem 6]{delima-roitman}), any  $(f_s,\phi)$-graph $\Sigma$
has the $T$-\emph{property}, meaning that
$T$ is a principal direction at any point of $\Sigma$.
More precisely,
\begin{equation}\label{eq-principaldirection}
AT=\frac{\phi''}{(\sqrt{1+(\phi')^2})^3}T.
\end{equation}

Conversely, any hypersurface of $M\times I$ with non vanishing angle function having $T$ as a principal direction
is given locally as an $(f_s,\phi)$-graph.

Given an 
$(f_s,\phi)$-graph $\Sigma,$
let $\{X_1\,,\dots ,X_n\}$ be
the orthonormal frame of principal directions of $\Sigma$
in which $X_n=T/\|T\|.$ In this case, 
for $1\le i\le n-1,$
the fields $X_i$  are all horizontal, that is, tangent to $M$ (cf. \eqref{eq-sum}).
Therefore, setting
\begin{equation}\label{eq-rho}
\rho:=\frac{\phi'}{\sqrt{1+(\phi')^2}}
\end{equation}
and considering \eqref{eq-normal}, we have, for all $i=1,\dots ,n-1,$ that
\[
k_i=\langle AX_i,X_i\rangle=-\langle\overbar\nabla_{X_i}N,X_i\rangle=\rho\langle\overbar\nabla_{X_i}\eta_s,X_i\rangle=-\rho k_i^s,
\]
where $k_i^s$ is the $i$-th principal curvature of $f_s\,.$ Also,
it follows from \eqref{eq-principaldirection}
that $k_n=\rho'.$ Thus, the array of principal curvatures of the $(f_s,\phi)$-graph $\Sigma$ is
\begin{equation}\label{eq-principalcurvatures}
k_i=-\rho k_i^s \,\, (1\le i\le n-1) \quad\text{and}\quad k_n=\rho'.
\end{equation}

The function defined in \eqref{eq-rho} --- to be called the $\rho$-\emph{function} of
the $(f_s,\phi)$-graph $\Sigma$ --- will play a major role in the sequel.
We remark that, up to a constant, the $\rho$-function of $\Sigma$ determines its
$\phi$-function.
Indeed, it follows from equality \eqref{eq-rho} that
\begin{equation}\label{eq-phi10}
\phi(s)=\int_{s_0}^{s}\frac{\rho(u)}{\sqrt{1-\rho^2(u)}}du+\phi(s_0), \,\,\, s_0, s\in I_0.
\end{equation}

We conclude this section by introducing  a special type of family of parallel hypersurfaces which will
be of special interest to us.

\begin{definition}
We call a  family $\{f_s:M_0\rightarrow M\,;\, s\in I_0\}$
of parallel hypersurfaces  \emph{isoparametric} if,
for each $s\in I_0,$  $f_s$ has constant mean curvature (depending on $s$).
If so, each hypersurface $f_s$ is also called \emph{isoparametric.}
\end{definition}

\section{Three Key Lemmas} \label{sec-lemmas}

The proofs of our main theorems will rely on the following three lemmas. The first
establishes conditions on the $\rho$-function of an $(f_s,\phi)$-graph under which
the resulting hypersurface is totally umbilical. The second states that any nontrivial
totally umbilical hypersurface of $M\times I$ is represented locally as such a graph.
In the third, we prove the convergence of a certain improper integral
which will be recurrent in the paper.

\begin{lemma} \label{lem-parallel}
Let $\{f_s:M_0\rightarrow M\,;\, s\in I_0\}$ be an
isoparametric family of totally umbilical hypersurfaces of $M$, and denote by $\lambda=\lambda(s)$ the
umbilical constant  of
$f_s$. Let $\rho$ be a solution of the
differential equation
\begin{equation}  \label{eq-difequation}
y'+\lambda(s)y=0, \,\,\, s\in I_0.
\end{equation}
Then, if  \,$0<\rho<1,$  the
$(f_s,\phi)$-graph $\Sigma$ determined by $\rho$ is a totally
umbilical hypersurface of the product $M\times\R.$
\end{lemma}

\begin{proof}
Let $\rho$ be a solution of \eqref{eq-difequation} satisfying $0<\rho<1.$
From \eqref{eq-principalcurvatures},
we have that the first $n-1$ principal curvatures of the $(f_s,\phi)$-graph $\Sigma$ determined by $\rho$
are all equal to $-\lambda(s)\rho(s).$ However,
$k_n=\rho'=-\lambda\rho,$ which implies that
$\Sigma$ is totally umbilical in $M\times\R.$
\end{proof}

Regarding equation \eqref{eq-difequation}, recall that
its general solution is
\begin{equation}\label{eq-solutiondifequation}
\rho(s)=c\exp\left(-\int_{s_0}^s\lambda(u)du\right), \,\,\, s_0\,, s\in I_0, \,\,c\in\R.
\end{equation}

\begin{lemma}  \label{lem-localgraph}
Let $\Sigma\subset M\times I$ be a  totally umbilical hypersurface with no horizontal points and
non vanishing angle function. Then, $\Sigma$ is given locally by an $(f_s,\phi)$-graph,
where $\{f_s:M_0\rightarrow M\,;\, s\in I_0\}$
is an isoparametric family of  totally umbilical hypersurfaces of $M.$ Moreover,
the following assertions are equivalent:
\begin{itemize}[parsep=1ex]
\item[\rm i)] $\Sigma$ is totally geodesic.
\item[\rm ii)] The angle function of $\Sigma$ is locally constant.
\item[\rm iii)] $f_s$ is totally geodesic for all $s\in I_0.$
\end{itemize}
\end{lemma}

\begin{proof}
Since $\Sigma$ is umbilical with no horizontal points, it has the $T$-property.
So, by \cite[Theorem 6]{delima-roitman}, $\Sigma$ is locally an $(f_s,\phi)$-graph.
Denoting  by $k(x)$ the principal curvature of $\Sigma$ at $x=(f_s(p),\phi(s)),$ we
have from equalities \eqref{eq-principalcurvatures} that
\[
\rho'(s)=k(x)=-\rho(s)k_i^s(p) \,\, \forall i=1,\dots n-1.
\]
Thus, for all $p\in M_0$ and $i=1,\dots ,n-1,$ the following equality holds:
\[
k_i^s(p)=-\frac{\rho'(s)}{\rho(s)},
\]
which implies that $f_s$ is totally umbilical and isoparametric.

Regarding the assertions (i)--(iii), we have just to notice that, since
$k_n=\rho',$ $\Sigma$ is totally geodesic if and only if $\rho$ is constant.
From this, the last equality above,
and the identities \eqref{eq-thetaparallel} and \eqref{eq-rho},
one easily concludes that (i), (ii) and (iii) are  equivalent.
\end{proof}

\begin{lemma} \label{lem-convergenceintegral}
Let $\rho:[a,b]\rightarrow\R$ be a differentiable function such
that $0\le \rho\le 1.$ Assume that either of the  following hold:
\begin{itemize}[parsep=1ex]
\item[\rm i)] $\rho(a)\ne 1$, $\rho(b)=1,$  and \,\,$\rho'(b)>0$.
\item[\rm ii)]$\rho(a)=1$, $\rho(b)\ne 1,$ and \,\,$\rho'(a)<0$.
\end{itemize}
Under any of these conditions,  the improper integral
\[
\int_{a}^{b}\frac{\rho(s)}{\sqrt{1-\rho^2(s)}}ds
\]
is convergent.
\end{lemma}
\begin{proof}
Assume that  (i) occurs. In this case, there exist positive constants, $\delta$ and $C,$ such that
$\rho'(s)\ge C>0\,\forall s\in(b-\delta ,b).$ Therefore,
\begin{eqnarray}
\int_{b-\delta}^{b}\frac{\rho(s)ds}{\sqrt{1-\rho^2(s)}} & \le & \int_{b-\delta}^{b}\frac{\rho'(s)ds}{\rho'(s)\sqrt{1-\rho^2(s)}}
\le\frac{1}{C} \int_{\rho(b-\delta)}^{1}\frac{d\rho}{\sqrt{1-\rho^2}} \nonumber\\
            & = & \frac{1}{C}\left(\frac{\pi}{2}-\arcsin(\rho(b-\delta))\right)\le \frac{\pi}{2C}\,, \nonumber
\end{eqnarray}
which proves the case (i). The proof for the case (ii) is analogous.
\end{proof}

\section{Totally Umbilical Hypersurfaces of $M\times I$ and $I\times_\omega M.$} \label{sec-umbilicalMxI}

In this section, we give a complete characterization
of totally umbilical hypersurfaces of the product spaces $M\times I$ and $I\times_\omega M.$
To begin with, let us observe that
Lemmas \ref{lem-parallel} and \ref{lem-localgraph} immediately give the following result.

\begin{theorem} \label{th-classifMxR}
There exist nontrivial totally  umbilical (resp. totally geodesic)
hypersurfaces in $M\times I$ if and only if
$M$ admits  isoparametric families of  totally umbilical (resp. totally geodesic) hypersurfaces.
Furthermore, if $\Sigma\subset M\times\R$ is totally umbilical and connected, then $\Sigma$ is
totally geodesic if and only if its angle function is constant.
\end{theorem}

It is a well known fact that $\q$ does not have
isoparametric families of totally geodesic hypersurfaces. Hence, we have

\begin{corollary}
A totally geodesic hypersurface of \,$\q\times\R$ is necessarily trivial.
\end{corollary}

If we let $M$ be a warped product $I_0\times_\omega M_0^{n-1}$,  then
the  standard isometric immersions $f_s: M_0\hookrightarrow \{s\}\times_\omega M_0\subset M$, $s\in I_0,$
define a family  $\mathscr F=\{f_s\,;\, s\in I_0\}$ of  parallel
and totally umbilical hypersurfaces of $M$ whose constant umbilical function
is $\omega'/\omega$ (see Section \ref{sec-preliminaries}).
In particular, $\mathscr F$ is isoparametric. Hence,
from Theorem \ref{th-classifMxR}, there exist non trivial totally umbilical hypersurfaces in $M\times\R.$

Conversely, assume that $\mathscr F=\{f_s:M_0\rightarrow M'\subset M\,;\, s\in I_0\}$
is an isoparametric family of totally umbilical
hypersurfaces covering an open set $M'\subset M.$ In this case, denoting by $\lambda(s)$ the
umbilical constant of $f_s$\,, we have that $M'$ is isometric
to the warped product $I_0\times_\omega M_0$\,, where the warping function
$\omega$ satisfies $\lambda=\omega'/\omega$ (see, e.g., \cite[Section 4]{bolton}).

From these considerations and Theorem \ref{th-classifMxR}, we obtain the following result, which
was previously obtained in \cite{souam-veken}.

\begin{corollary}
There exists a nontrivial totally  umbilical
hypersurface $\Sigma$ in $M\times I$ if and only if
$M$ contains an open set which is isometric to  a warped product
$I_0\times_\omega M_0$ as above.  If so,
$\Sigma$ is totally geodesic if and only if the
corresponding warping function is constant.
\end{corollary}

Given a warped product $I\times_\omega M$,  consider the diffeomorphism
\begin{equation} \label{eq-F}
F:I\rightarrow J=F(I)
\end{equation}
such that $F'=1/\omega.$ As can be easily checked, the map
\begin{equation} \label{eq-varphi}
\begin{array}{cccc}
\varphi\colon & I\times_\omega M   & \rightarrow & M \times J\\[1ex]
              & (t,p) & \mapsto     & (p,F(t))
\end{array}
\end{equation}
is a conformal diffeomorphism whose conformal factor is $F'=1/\omega.$

It is a well known fact that the  property of being
totally umbilical is preserved under a conformal diffeomorphism
between the ambient spaces. Therefore,
a hypersurface $\Sigma\subset I\times_\omega M$ is totally umbilical if and only if
$\varphi(\Sigma)$ is totally umbilical in $M\times J.$ This fact, together with Theorem \ref{th-classifMxR},
gives the following characterization of umbilical hypersurfaces of $I\times_\omega M.$

\begin{theorem} \label{th-classifMxRwarped}
There exist nontrivial totally  umbilical
hypersurfaces in $I\times_\omega M$ if and only if
$M$ admits  isoparametric families of  totally umbilical hypersurfaces.
\end{theorem}

In the above theorem, as in the Riemannian product case, the condition on the existence
of isoparametric families of totally umbilical hypersurfaces
can be replaced by that of $M$ having
a local structure of a warped product.

\section{Totally Umbilical Hypersurfaces of  $\s^n\times\R$ and  $\h^n\times\R$} \label{sec-umbilicalQnxR}

Our aim in this section is to extend
the main results by Souam and Toubiana \cite{souam-toubiana} on totally umbilical hypersurfaces
of $\s^2\times\R$ and  $\h^2\times\R$ to the spaces  $\s^n\times\R$ and  $\h^n\times\R$.
Actually, our approach includes the case
$n=2$ and  provides new proofs for Souam and Toubiana's theorems.

\subsection{Totally Umbilical Hypersurfaces of $\s^n\times\R$} \label{sec-umbilicalSnxR}
In the next two theorems, we show that there exist unique (up to ambient isometries)
properly embedded nontrivial totally umbilical hypersurfaces of $\s^n\times\R$ which are all rotational.

\begin{theorem} \label{th-umbilicalspherical}
There exists a one-parameter family $\{\Sigma(c)\,;\, c>0\}$ of
totally umbilical hypersurfaces of \,$\s^n\times\R$
with the following properties:
\begin{itemize}[parsep=1ex]
\item[\rm i)] For any $c>0,$ $\Sigma(c)$ is complete, analytic, properly embedded and rotational.
\item[\rm ii)] For $c=1,$ $\Sigma(c)$ is an unbounded graph over an open hemisphere of \,$\s^n$. In particular,
$\Sigma(c)$ is homeomorphic to $\R^n.$
\item[\rm iii)] For $c\ne 1,$ $\Sigma(c)$ is homeomorphic to $\s^n,$ and is invariant under reflection
with respect to a horizontal hyperplane $\s^n\times\{t_0\}.$
\end{itemize}
\end{theorem}
\begin{proof}
The hypersurfaces $\Sigma(c)$ will be obtained by properly ``gluing''
$(f_s,\phi)$-graphs, except for case (ii), in which $\Sigma(c)$ is an $(f_s,\phi)$-graph itself.
To accomplish that,
let us fix $o\in\s^n$ and, for each $s\in (0,\pi/2),$ let
$f_s:\s^{n-1}\rightarrow\s^n$ be the geodesic sphere of $\s^n$ with center
at $o$ and radius $s.$ Then, $f_s$ is umbilical with principal curvature
\[
\lambda(s)=-\cot s
\]
with respect to the outward orientation.

Given $c>0,$ define the interval $I_c=[0,a),$ where $a\in (0,\pi/2]$ and
\[
a=\left\{
\begin{array}{lcc}
\pi/2 & \text{if} & c=1\\[.8ex]
\arcsin(1/c) & \text{if} & c>1 \\ [.8ex]
\arcsin(c) & \text{if} & c<1
\end{array}
\right..
\]
In this setting, the function $\rho_c:[0,a)\rightarrow\R$ given by
\begin{equation} \label{eq-rhoc002}
\rho_c(s)=\left\{
\begin{array}{lcc}
\sin(s) & \text{if} & c=1\\[.8ex]
c\sin(s) & \text{if} & c>1 \\ [.8ex]
\frac{1}{c}\sin(s) & \text{if} & c<1
\end{array}
\right.
\end{equation}
is clearly a solution of $y'=\cot(s)y$ which satisfies
\begin{equation}  \label{eq-rhoc01}
0<\rho_c|_{(0,a)}<1, \quad \rho_c(0)=0, \quad\text{and}\quad \lim_{s\rightarrow a}\rho_c(s)=1.
\end{equation}

Thus, by Lemma \ref{lem-parallel},  $\rho_c$ determines a totally umbilical
$(f_s,\phi)$-graph $\Sigma$ which is rotational, for the hypersurfaces $f_s$ are
concentric spheres.  Moreover, since $\rho$ is real analytic, by \eqref{eq-phi10}, so
is $\phi,$ which implies that the graph $\Sigma$ itself is analytic. In addition,
$\phi(0)=\phi'(0)=0$. So $o\in\Sigma$ is a minimum of the height function of $\Sigma$.

For $c=1$, we have $\rho_c(s)=\sin(s).$ In this case, it follows from  \eqref{eq-phi10} that
\[
\phi(s)=\log\left(\frac{1}{\cos (s)}\right), \,\,\, s\in I_c=[0,\pi/2).
\]
In particular, $\phi(s)\rightarrow +\infty$ as $s\rightarrow\pi/2,$
so that $\Sigma(1):=\Sigma$ is an unbounded
graph over the hemisphere of $\s^n$ centered at $o$ (Fig. \ref{fig-unbounded}).
This covers  (ii).

\begin{figure}[htbp]
\includegraphics{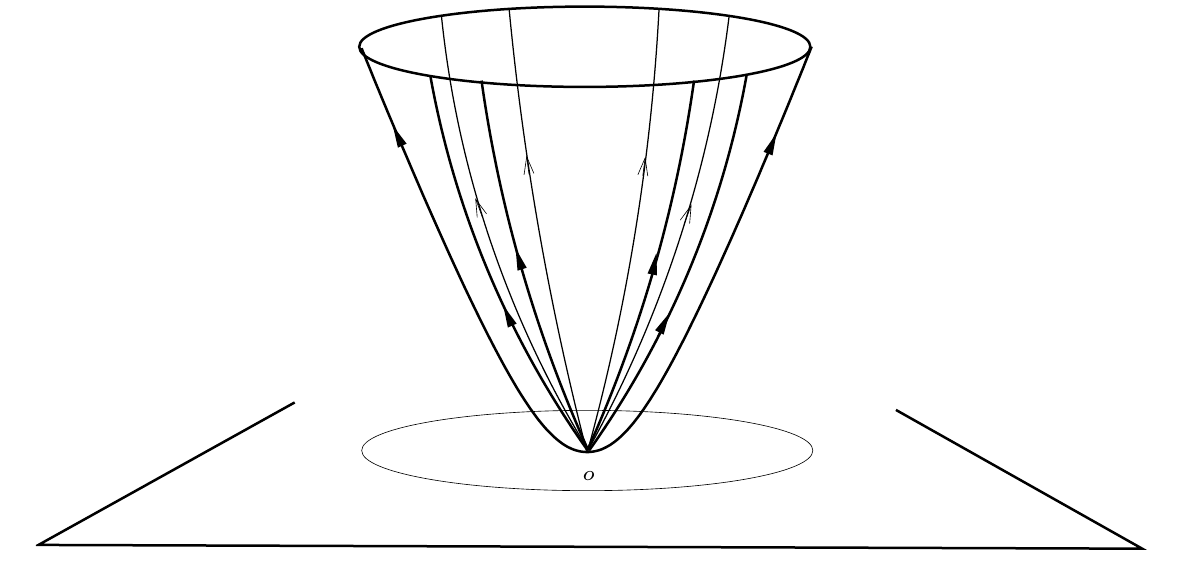}
\caption{\small An unbounded totally umbilical graph  in $\s^n\times\R.$}
\label{fig-unbounded}
\end{figure}


For $0<c\ne 1,$ we have  that
\[
\lim_{s\rightarrow a}\rho_c'(s)=c\cos(a)>0.
\]
So, by Lemma \ref{lem-convergenceintegral}, $\phi$ is bounded, that is, the limit
\[
t_0:=\lim_{s\rightarrow a}\phi_c(s)=\int_{0}^{a}\frac{\rho_c(s)}{\sqrt{1-(\rho_c(s))^2}}ds
\]
is well defined. Moreover, from \eqref{eq-phi10} and the last equality in \eqref{eq-rhoc01},  one has
\[
\lim_{s\rightarrow a}\phi'(s)=+\infty,
\]
which implies that the tangent spaces of the $(f_s,\phi)$-graph $\Sigma$ along its $T$-trajectories
converge to  vertical tangent spaces along the sphere $S_a(o)\times\{t_0\}$ (Fig. \ref{fig-sphere}).

Thus, denoting
by $\Sigma'$ the reflection of $\Sigma$ with respect to $\s^n\times\{t_0\},$ we have that
\[
\Sigma(c)={\rm closure}(\Sigma)\cup{\rm closure}(\Sigma')
\]
is a totally umbilical hypersurface of $\s^n\times\R$, which is analytic, properly embedded, homeomorphic to
$\s^n$, and symmetric with respect to $\s^n\times\{t_0\}.$ This shows (iii) and concludes our proof. 
\end{proof}

\begin{figure}[htbp]
\includegraphics{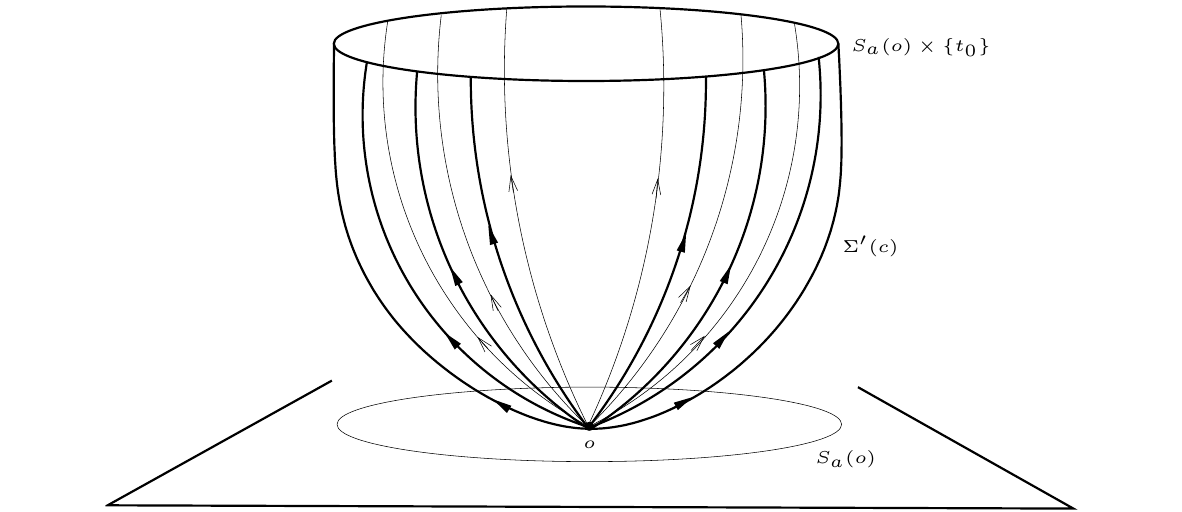}
\caption{\small All $T$-trajectories of   $\Sigma'(c)$ emanate from $o$ and meet its boundary
$S_a(o)\times\{t_0\}$ orthogonally.}
\label{fig-sphere}
\end{figure}

\begin{theorem} \label{th-uniqspherical}
Let $\Sigma$ be a connected totally umbilical and non totally geodesic hypersurface
of \,$\s^n\times\R.$ Then, $\Sigma$ is contained in one of the hypersurfaces $\Sigma(c)$
given in Theorem $\ref{th-umbilicalspherical}.$ 
\end{theorem}
\begin{proof}
Since $\Sigma$ is connected and non totally geodesic, neither $T$ nor $\theta$ can vanish on
an open set of $\Sigma.$ Hence, we can assume that $\theta T$ never vanishes on
$\Sigma.$ In this case, it follows from Lemma \ref{lem-localgraph} that $\Sigma$ is the union
of  $(f_s,\phi)$-graphs $\Sigma'$ whose parallel family $f_s$ is isoparametric and totally umbilical.
However, the only parallel family in $\s^n$ with these properties is that of concentric geodesic spheres, which
gives that $\Sigma'$ is rotational, and that the $\rho$-function of $\Sigma'$ is one the functions $\rho_c$ defined in
\eqref{eq-rhoc002}. Therefore, $\Sigma'\subset\Sigma(c).$ This fact and the connectedness of $\Sigma$
clearly imply  that $\Sigma\subset\Sigma(c).$
\end{proof}

\subsection{Totally Umbilical Hypersurfaces of $\h^n\times\R$}
The class of complete nontrivial totally umbilical hypersurfaces of $\h^n\times\R$ is richer than that
of \,$\s^n\times\R$. Indeed, as we shall see, besides the rotational ones, this class also includes
hypersurfaces which are invariant by  certain translational isometries of $\h^n\times\R.$

The reason for that abundance becomes clear when we consider, together with  Theorem \ref{th-classifMxR},
the 
fact that we have three types of  foliations  of $\h^n$ by isoparametric totally umbilical hypersurfaces.
Namely, by geodesic spheres, by horospheres, and by  equidistant hypersurfaces (from
a fixed totally geodesic hyperplane of $\h^n$).

In the next theorems, we show that  each of these foliations of $\h^n$
gives rise to  complete  nontrivial totally umbilical hypersurfaces of $\h^n\times\R$,
which are unique, and invariant by one of the aforementioned isometries of this space.

\begin{theorem}\label{th-umbilicalhyperbolic}
There exists a one-parameter family $\{\Sigma(c)\,;\, c>0\}$ of
rotational totally umbilical hypersurfaces of \,$\h^n\times\R$ which are complete, analytic and properly embedded.
Furthermore, for any $c>0,$ $\Sigma(c)$ is homeomorphic to $\s^n,$ and is invariant under reflection
with respect to a hyperplane $\h^n\times\{t_0\}.$
\end{theorem}
\begin{proof}
Let us fix a point $o\in\h^n$ and, for each $s\in (0,+\infty),$ let
$f_s:\s^{n-1}\rightarrow\h^n$ be the geodesic sphere of $\h^n$ with center
at $o$ and radius $s.$ Then, $f_s$ is umbilical and
\[
\lambda(s)=-\coth s
\]
is its umbilical function with respect to the outward orientation.

Given $c>0,$ define the interval $I_c=[0,a),$ where $a=\sinh^{-1}(1/c).$
In this setting, it is clearly seen that the function 
\[
\rho_c(s)=c\sinh(s), \,\,\, s\in I_c=[0, a),
\]
is  a solution of $y'=\coth(s)y$ which satisfies
\begin{equation}  \label{eq-rhoc}
0<\rho_c|_{(0,a)}<1, \quad \rho_c(0)=0, \quad\text{and}\quad \lim_{s\rightarrow a}\rho_c(s)=1.
\end{equation}
So, by Lemma \ref{lem-parallel},  $\rho_c$ determines
a totally umbilical $(f_s,\phi)$-graph $\Sigma$ in $\h^n\times\R.$
As in the proof of Theorem \ref{th-umbilicalspherical}, $\Sigma$ is rotational and analytic. In addition,
\[
\lim_{s\rightarrow a}\rho_c'(s)=c\cosh(a)>0,
\]
which, together with Lemma \ref{lem-convergenceintegral}, gives that $\phi$ is bounded.

Henceforth, arguing just as in the
proof of Theorem \ref{th-umbilicalspherical}, one easily obtains
the desired hypersurface $\Sigma(c)$  by reflecting $\Sigma$ with respect to the horizontal hyperplane
$P_{t_0}$\,, where $t_0$ is the limit of $\phi(s)$ as $s\rightarrow a.$
\end{proof}

\begin{theorem}\label{th-horospheretype}
There exists a complete
totally umbilical hypersurface $\Sigma$ of \,$\h^n\times\R$ which is analytic, properly embedded and  homeomorphic to $\R^n.$
Furthermore,  $\Sigma$ is foliated by horizontal horospheres,
is symmetric with respect to $\h^n\times\{0\}$,
is contained in the slab \,$\h^n\times(-\frac{\pi}{2},\frac{\pi}{2})$,
and is asymptotic to the hyperplanes  $P_{-\pi/2}$ and $P_{\pi/2}$. 
\end{theorem}

\begin{proof}
Let $\{f_s:\R^{n-1}\rightarrow\h^n\,;\, s\in (-\infty,+\infty)\}$  be a
family of parallel horospheres of $\h^n.$ Each $f_s$ is totally umbilical with
constant umbilical function equal to $1.$ So, the associated differential equation \eqref{eq-difequation}
in this case is $y'=-y.$ Clearly, the function
\[
\rho(s)=e^{-s},\,\,\, s\in(0,+\infty),
\]
is a solution of  this equation which satisfies (see Remark \ref{rem-horospheretype} below):
\begin{equation} \label{eq-rhohorosphere}
0<\rho<1, \quad \lim_{s\rightarrow 0}\rho(s)=\rho(0)=1, \quad\text{and}\quad \lim_{s\rightarrow 0}\rho'(s)=-\rho(0)=-1.
\end{equation}
In addition, from the last equality in \eqref{eq-rhohorosphere} and Lemma \ref{lem-convergenceintegral},
the function
\[
\phi(s)=\int_{a}^{s}\frac{\rho(u)}{\sqrt{1-(\rho(u))^2}}du, \,\, \,\,\,  s\in(0,+\infty),
\]
is well defined, i.e., this improper integral is convergent.

Therefore, by Lemma \ref{lem-parallel}, the correspondent
$(f_s,\phi)$-graph $\Sigma'$ is totally umbilical in
$\h^n\times\R$.
Notice that, identifying $\h^n\times\{0\}$ with $\h^n,$ we have that
$\Sigma'$ is a graph over $\h^n- B$ with boundary $\mathscr H$\,, where
$\mathscr H$ is the horosphere $f_0(\R^{n-1})$ and $B$ is the horoball bounded by
$\mathscr H$ (Fig. \ref{fig-Horograph}).  Also, since $\rho$ is analytic, so is $\Sigma'$\,.

\begin{figure}[htbp]
\includegraphics{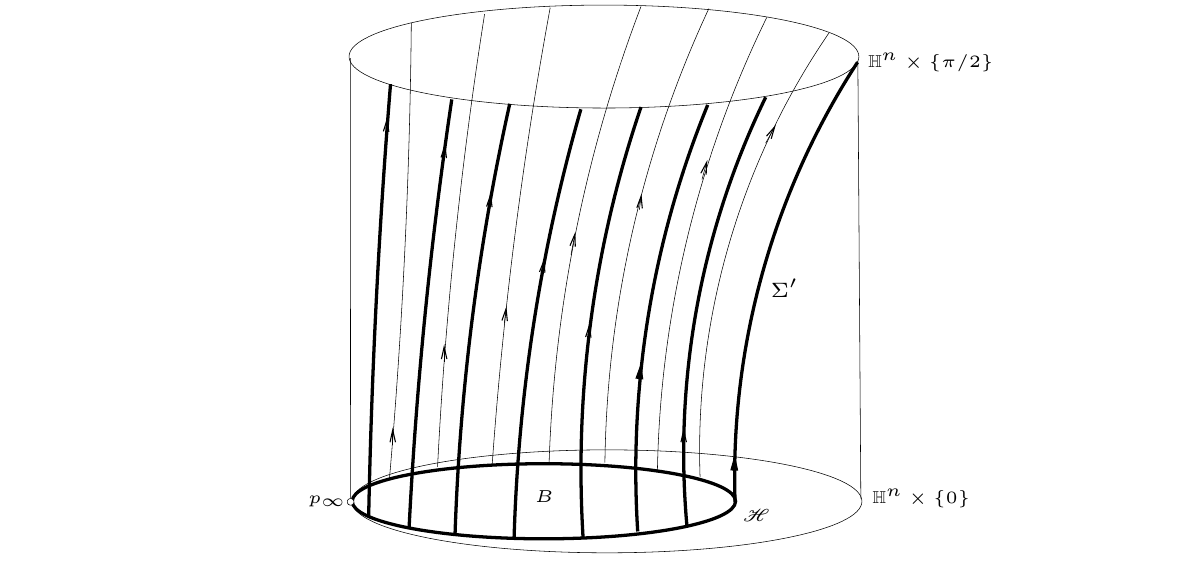}
\caption{\small All $T$-trajectories of $\Sigma'$
emanate from $\mathscr H$ orthogonally and proceed   to the boundary at infinity of
the hyperplane $\h^n\times\{\pi/2\}.$}
\label{fig-Horograph}
\end{figure}

We claim  that the function $\phi$ is bounded above by $\pi/2$ and that
$\Sigma'$ is asymptotic to the hyperplane $P_{\pi/2}=\h^n\times\{\pi/2\}.$
Indeed, since $\rho'=-\rho,$ one has
\[
\phi(s) = \int_{0}^{s}\frac{\rho(u)}{\sqrt{1-\rho^2(u)}}du = -\int_{1}^{\rho(s)}\frac{d\rho}{\sqrt{1-\rho^2}}=
\frac{\pi}{2}-\arcsin\rho(s),
\]
which proves the claim, for $\rho(s)\rightarrow 0$ as $s\rightarrow +\infty.$

Finally, since $\rho(0)=1,$  as in the previous theorems, we have that
the tangent spaces of the closure of $\Sigma'$ along its boundary $\mathscr H$ are all vertical.
Consequently, setting $\Sigma''$ for the reflection of $\Sigma'$ with respect to
$\h^n\times\{0\}$, and defining
\[
\Sigma:={\rm closure}\,(\Sigma')\cup {\rm closure}\,(\Sigma''),
\]
we can argue  just as before
and conclude that $\Sigma$ is a complete,  analytic,  properly embedded, and  totally umbilical
hypersurface of $\h^n\times\R$ which is homeomorphic to $\R^n.$ It is also clear from the construction that
$\Sigma$ is
foliated by horizontal horospheres,  is
symmetric with respect to $\h^n\times\{0\},$ and is asymptotic
to the hyperplanes  $P_{-\pi/2}$ and  $P_{\pi/2}$,
as we wished to prove.
\end{proof}

\begin{remark} \label{rem-horospheretype}
In the above proof,  taking the general solution
$\rho_c(s)=ce^{-s}$ of the equation $y'=-y$, and proceeding analogously, one obtains
a one-parameter family of \emph{isometric} totally umbilical hypersurfaces of $\h^n\times\R,$ since
the horospheres of $\h^n$ are pairwise isometric. Thus, only the case $c=1$ had to be considered.
\end{remark}

\begin{theorem}\label{th-equidistanttype}
There exists a one-parameter family $\{\Sigma(c)\,;\, 0<c<1\}$ of
totally umbilical hypersurfaces of \,$\h^n\times\R$ which are analytic,
properly embedded and  homeomorphic to $\R^n.$
Furthermore, for any $c>0,$ $\Sigma(c)$ is foliated by equidistant hypersurfaces, is periodic in
the vertical direction, and is
symmetric with respect to a discrete set of  horizontal hyperplanes.
\end{theorem}

\begin{proof}
Let $f_0:\R^{n-1}\rightarrow\h^n$ be a totally geodesic hyperplane of $\h^n$, and
let
\[
\mathscr F:=\{f_s:\R^{n-1}\rightarrow\h^n\,;\, s\in (-\infty,+\infty)\}
\]
be the parallel family of equidistant hypersurfaces of $f_0$
in  $\h^n.$ Each $f_s$ is totally umbilical with
constant umbilical function
\[
\lambda(s)=-\tanh(s), \,\,\, s\in (-\infty, +\infty).
\]
Thus, the associated differential equation \eqref{eq-difequation}
in this case is $y'=\tanh(s)y.$

Given $0<c<1$, let  $a>0$ be such that $\cosh (a)=1/c.$ In this setting,
\[
\rho(s)=c\cosh(s),\,\,\, s\in(-a, a),
\]
is a solution of $y'=\tanh(s)y$ which satisfies
\begin{equation} \label{eq-rhoequidistant}
0<\rho<1, \quad \lim_{s\rightarrow\pm a}\rho(s)=\rho(\pm a)=1, \quad\text{and}\quad \lim_{s\rightarrow\pm a}\rho'(s)=c\sinh (\pm a)\ne 0.
\end{equation}
In addition, from the last equality in \eqref{eq-rhoequidistant} and Lemma \ref{lem-convergenceintegral},
the function
\[
\phi(s)=\int_{-a}^{s}\frac{\rho(u)}{\sqrt{1-(\rho(u))^2}}du, \,\, \,\,\, s\in (-a,a),
\]
is well defined and is also bounded.

Therefore, by Lemma \ref{lem-parallel}, the correspondent
$(f_s,\phi)$-graph $\Sigma'(c)$ is totally umbilical in
$\h^n\times\R$.
Identifying $\h^n\times\{0\}$ with $\h^n,$ we have that
$\Sigma'(c)$ is a graph over the open convex region $\Omega_a$ of $\h^n$ which is
bounded by the equidistant hypersurfaces
$\mathscr E_{-a}:=f_{-a}(\R^{n-1})$ and $\mathscr E_{a}:=f_{a}(\R^{n-1})$\,. In particular,
$\Sigma'(c)$ is homeomorphic to $\R^n$ and has boundary
$\partial\Sigma'(c)=(\mathscr E_{-a}\times\{0\})\cup(\mathscr E_{a}\times\{\phi(a)\})$
(Fig. \ref{fig-equidistant}).
Also, as in the preceding theorems,
$\Sigma'(c)$ is analytic, as a consequence of  $\rho$ being analytic.

Now, we have just to observe that, by the second equality in \eqref{eq-rhoequidistant},
the tangent spaces of $\Sigma'(c)$ are vertical
along its boundary.
Therefore, we obtain the hypersurface $\Sigma(c)$, as stated,  by
continuously reflecting $\Sigma'(c)$  with respect to the
horizontal hyperplanes $\h^n\times \{k\phi (a)\}, \, k\in\mathbb{Z}.$
\end{proof}

\begin{figure}[htbp]
\includegraphics{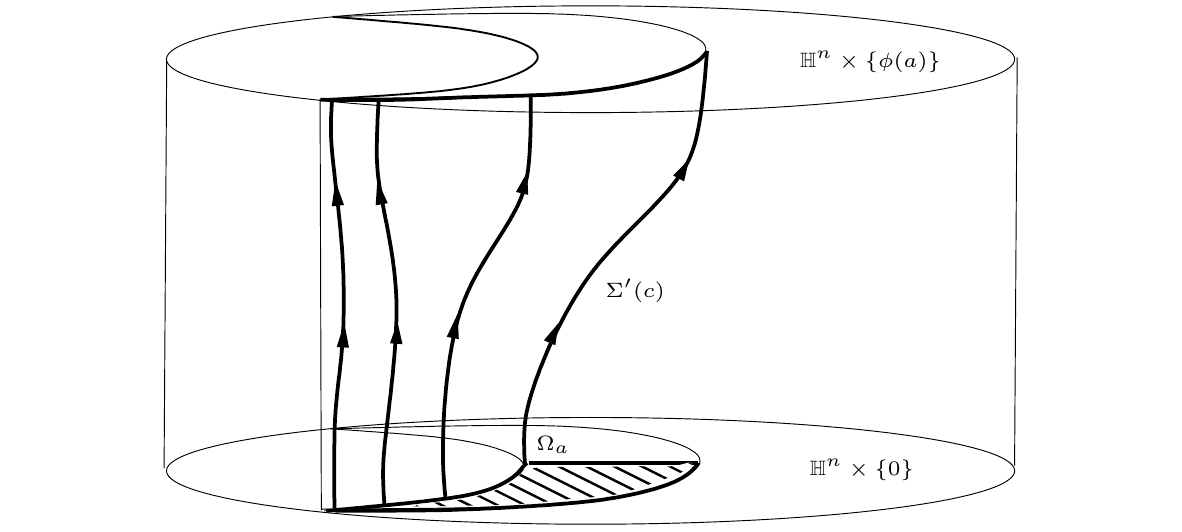}
\caption{\small Half of the graph $\Sigma'(c)$. The $T$-trajectories emanate  from $\mathscr E_{-a}\times\{0\}$ orthogonally,
and meet $\mathscr E_{a}\times\{\phi(a)\}$ orthogonally as well.}
\label{fig-equidistant}
\end{figure}

\begin{theorem} \label{th-uniqhyperbolic}
Let $\Sigma$ be a connected totally umbilical and non totally geodesic hypersurface
of \,$\h^n\times\R.$ Then, $\Sigma$ is contained in one of the complete hypersurfaces
given in Theorems \ref{th-umbilicalhyperbolic}--\ref{th-equidistanttype}.
\end{theorem}
\begin{proof}
We have just to observe that the only totally umbilical isoparametric hypersurfaces of
hyperbolic space $\h^n$ are geodesic spheres, horospheres, and equidistant hypersurfaces, which are
the very horizontal sections of the complete hypersurfaces constructed
in Theorems \ref{th-umbilicalhyperbolic}--\ref{th-equidistanttype}.
Hence, we get to the conclusion by applying  the same reasoning as in
the proof of Theorem \ref{th-uniqspherical}.
\end{proof}

In hyperbolic space $\h^n,$  there are three special
types of one-parameter families of
isometries. Namely,  the rotations around a fixed point (\emph{elliptic isometries}), the translations
along a fixed geodesic (\emph{hyperbolic isometries}), and the translations  along  horocycles
sharing the same point at infinity (\emph{parabolic isometries}).
It is well known that $SO(n)$ is the group of elliptic isometries of $\h^n$.
In the upper half-space model of $\h^n,$ the hyperbolic isometries are the Euclidean homotheties
from the origin, whereas parabolic isometries are horizontal Euclidean translations along fixed horizontal directions
(see, e.g., \cite[Section 10.1]{lopez}).

The isometries of $\h^n,$ in particular those described above,
extend  to isometries of $\h^n\times\R$ which  fix the factor $\R$ pointwise.
Of course, the same applies to $\s^n$, whose isometries are all elliptic.
More precisely, given
$\Phi_0\in{\rm Isom}\,(\q),$ the map
\[
\Phi(p,t)=(\Phi_0(p),t), \,\,\, (p,t)\in\q\times\R,
\]
is  an isometry of $\q\times\R.$ We call $\Phi$ \emph{elliptic, parabolic or hyperbolic} according
to wether $\Phi_0$ is elliptic, parabolic or hyperbolic.

An analogous consideration can be made in a  warped product $I\times_\omega\q,$ so that
elliptic, parabolic and hyperbolic isometries are also defined in this context.

\begin{definition} \label{def-symmetric}
A hypersurface $\Sigma$ of $\q\times\R$ or $I\times_\omega\q$ which is invariant by
an elliptic, parabolic or hyperbolic isometry will be called \emph{symmetric}.
\end{definition}

Rotational hypersurfaces are obviously invariant by elliptic isometries, which is the case of
the hypersurfaces $\Sigma(c)$ of Theorems \ref{th-umbilicalspherical} and \ref{th-umbilicalhyperbolic}.
The hypersurface $\Sigma$ of Theorem \ref{th-horospheretype} is invariant by parabolic isometries,
since its horizontal sections project vertically to families of parallel horospheres of $\h^n.$ Analogously,
the hypersurfaces $\Sigma(c)$ of Theorem \ref{th-equidistanttype} are invariant by hyperbolic isometries.
These facts, together with Theorems \ref{th-uniqspherical} and \ref{th-uniqhyperbolic}, give the following
result.

\begin{theorem} \label{th-classification}
All hypersurfaces  of Theorems \ref{th-umbilicalspherical}, \ref{th-umbilicalhyperbolic},  \ref{th-horospheretype}
and \ref{th-equidistanttype} are symmetric. Consequently, any connected nontrivial  totally umbilical hypersurface
of \,$\q\times\R$ is an open set of a properly embedded symmetric hypersurface.
\end{theorem}

We should mention that, in the spirit of the above theorem,
a local description of the totally umbilical hypersurfaces of $\s^n\times\R$ and $\h^n\times\R$ by means
of rotating curves was given in \cite{vekenetal1,vekenetal}.

\section{Totally Umbilical Hypersurfaces of $I\times_\omega\q.$} \label{sec-Umbilicalwarp}
In what follows, we extend the results of the previous section to the context of
warped products $I\times_\omega\q,$ where $I\subset\R$ is an open interval.
In this setting,  let us  consider
again the conformal diffeomorphism
\begin{equation} \label{eq-varphi}
\begin{array}{cccc}
\varphi\colon & I\times_\omega\q   & \rightarrow & \q\times J\\[1ex]
              & (t,p) & \mapsto     & (p,F(t))
\end{array},
\end{equation}
where $F$ is the diffeomorphism
\begin{equation} \label{eq-F}
F:I\rightarrow J=F(I), \,\,\, F'=1/\omega.
\end{equation}
We will assume, without loss of generality,
that
\begin{equation} \label{eq-Jinterval}
J=(-\delta,\delta),  \,\,\, 0<\delta\le +\infty.
\end{equation}

As we have already pointed out,  the  property of being
totally umbilical is preserved by the diffeomorphism $\varphi.$
Furthermore, since the conformal factor of $\varphi$ depends only on $t,$
and $\varphi(\{t\}\times_\omega\q)=\q\times\{F(t)\},$ we have that
$\{t\}\times_\omega\q$ is homothetic to $\q\times\{F(t)\}.$
From these considerations, we get the following result.



\begin{lemma} \label{lem-varphi}
Let $\varphi$ be the conformal diffeomorphism defined in \eqref{eq-varphi}.
Given a hypersurface $\Sigma$ in $I\times_\omega\q,$ the following assertions hold:
\begin{itemize}[parsep=1ex]
\item[\rm i)] $\Sigma$ is totally umbilical in $I\times_\omega\q$  if and only if $\varphi(\Sigma)$ is totally umbilical in $\q\times J.$
\item[\rm ii)] A vertical section $\Sigma_t\subset\Sigma$ is a sphere
(resp. horosphere, equidistant hypersurface) of $\{t\}\times_\omega\q$ if and only if
$\varphi(\Sigma_t)$ is a horizontal section of $\varphi(\Sigma)$ which is a sphere
(resp. horosphere, equidistant hypersurface) of $\q\times\{F(t)\}.$ Consequently,
$\Sigma$ is symmetric in $I\times_\omega\q$ if and only if $\varphi(\Sigma)$ is symmetric in
$\q\times J.$
\end{itemize}
\end{lemma}

Now, we are in position to construct and classify nontrivial totally umbilical
hypersurfaces of $I\times_\omega\q.$ Firstly, let us recall that
the \emph{vertical diameter} of a compact hypersurface of $M\times\R$ is
defined as the difference between the maximum and minimum values of its  height function $\xi.$

\begin{theorem} \label{th-warprotacional}
Given a warping function $\omega,$ there exists a one-parameter family $\{\Sigma(c)\,;\, c>0\}$ of
embedded rotational totally umbilical hypersurfaces of $I\times_\omega\q$ with the following properties:
\begin{itemize}[parsep=1ex]
\item[\rm i)] For $\epsilon=1$ and $c=1,$  $\Sigma(c)$ is a graph over an open convex geodesic ball of \,$\s^n$.
In particular, $\Sigma$ is homeomorphic to $\R^n$.
\item[\rm ii)] For $\epsilon=1$ and $c\ne 1,$ or $\epsilon=-1$ and $c>0$, $\Sigma(c)$ is either homeomorphic to an
$n$-annulus or to the $n$-sphere $\s^n.$ The latter occurs if and only if
the vertical diameter of $\varphi(\Sigma(c))$ is less than $2\delta,$  where
$\delta$ is as in \eqref{eq-Jinterval}.
\end{itemize}
\end{theorem}

\begin{proof}
(i) Let $\varphi$ be the map defined in \eqref{eq-varphi}, and
denote by $\Sigma_0$  the part of the totally umbilical graph of $\s^n\times\R$
obtained in Theorem \ref{th-umbilicalspherical}-(ii) which
is contained in $\s^n\times J.$ Since the whole graph is
defined over an open hemisphere of $\s^n$ centered at a point $o\in\s^n,$
we have that $\Sigma_0$ is defined on a maximal convex open
ball $B\subset\s^n$ centered at $o.$
Therefore, by Lemma \ref{lem-varphi},  $\Sigma:=\varphi^{-1}(\Sigma_0)$
is a rotational totally umbilical graph
in $I\times_\omega\s^n$ over $B.$

\vt

\noindent
(ii) We consider only the case $\epsilon=1,$ since the proof for the case $\epsilon=-1$ is completely analogous.
Let $\{\Sigma_0(c)\,;\, c>0\}$ be the family of totally umbilical spheres of $\s^n\times\R$
obtained in Theorem \ref{th-umbilicalspherical}-(iii).  Given $c>0,$ after a vertical translation, we can assume that
$\Sigma_0(c)$ is symmetric with respect to
the  hyperplane $\s^n\times\{0\}.$ In this way, \[\Sigma_0'(c):=\Sigma_0(c)\cap(\s^n\times(-\delta,\delta))\] is
clearly homeomorphic to an $n$-annulus if the vertical diameter
of $\Sigma_0(c)$ is at least $2\delta.$ Otherwise, $\Sigma_0'(c)$ coincides with
$\Sigma_0(c)$. Therefore,
$\Sigma(c):=\varphi^{-1}(\Sigma_0'(c))$ is
an embedded totally umbilical hypersurface of $I\times_\omega\s^n$
which is homeomorphic to an $n$-annulus
or to the $n$-sphere $\s^n.$ This finishes the proof.
\end{proof}

\begin{theorem}\label{th-warpequidistant-horosphere}
Given a warped product $I\times_\omega\h^n,$ the following hold:
\begin{itemize}[parsep=1ex]
\item[\rm i)] There exists an
embedded totally umbilical hypersurface $\Sigma$ of $I\times_\omega\h^n$ which
is foliated by horospheres and is homeomorphic to $\R^n$. In addition,
$\Sigma$ is complete, provided  $\delta>\pi/2.$
\item[\rm ii)] There exists a one-parameter family $\{\Sigma(c)\,;\, 0<c<1\}$ of
embedded totally umbilical hypersurfaces of $I\times_\omega\h^n$ such that each member
$\Sigma(c)$
is foliated by equidistant hypersurfaces and is homeomorphic to $\R^n$.
\end{itemize}
\end{theorem}

\begin{proof}
We only show (i), since the proof of (ii) is analogous.
As in the preceding proof,
consider the map $\varphi$ defined in \eqref{eq-varphi}. Denote by
$\Sigma_0$  the part of the totally umbilical hypersurface of $\h^n\times\R$
obtained in Theorem \ref{th-horospheretype} which is contained in $\h^n\times (-\delta,\delta).$
Clearly, $\Sigma_0$ is homeomorphic to $\R^n.$ Thus,
by Lemma  \ref{lem-varphi}, the hypersurface $\Sigma:=\varphi^{-1}(\Sigma_0)$ is
embedded and totally umbilical in $I\times_\omega\h^n$. In addition,
$\Sigma$ is homeomorphic to $\R^n$ and is foliated by horospheres.

Finally, if $\delta>\pi/2$, by Theorem \ref{th-horospheretype},  $\Sigma_0$ is complete and
bounded away from the boundary of $\h^n\times (-\delta,\delta).$ Thus,
$\Sigma$ is bounded away from the boundary of $I\times_\omega\h^n$,
which implies that it is complete. This finishes the proof.
\end{proof}

The above results, together with
Theorem \ref{th-classification},
give the following characterization of totally umbilical hypersurfaces of
warped products $I\times_\omega\q.$

\begin{theorem} \label{th-warpedclassification}
Any connected nontrivial totally umbilical hypersurface
of $I\times_\omega\q$ is necessarily an open set of an embedded symmetric
totally umbilical hypersurface.
\end{theorem}

To illustrate Theorems \ref{th-warprotacional} and \ref{th-warpequidistant-horosphere},
 let us consider the warping functions
\[
\omega_1(t)=t, \, t\in I_1=(0,+\infty), \,\,\, \text{and} \,\,\, \omega_2(t)=e^{-t}, \, t\in I_2=(-\infty,+\infty).
\]
The associated functions $F_1$\,, $F_2$ satisfying $F_i'=1/\omega_i$ are
\[
F_1(t)=\log t, \, t\in I_1, \,\,\,\, \text{and} \,\,\,\,  F_2(t)=e^{t}, \, t\in I_2.
\]
Since the range of both $F_1$ and $F_2$ is $\R,$ and $I_1$ and $I_2$ are unbounded,
the totally umbilical hypersurfaces of $I_i\times_{\omega_i}\q$
from Theorems \ref{th-warprotacional} and \ref{th-warpequidistant-horosphere} are all complete,
except for the case (ii) in Theorem \ref{th-warpequidistant-horosphere} applied to $I_1\times_{\omega_1}\q$.



\begin{thebibliography}{99}

\bibitem{bishop-oneill} Bishop, R. L., O'Neill, B.:  Manifolds of negative curvature.
Trans. Amer. Math. Soc. {\bf 145}, 1--49 (1969).


\bibitem{bolton} Bolton, J.: On Riemannian manifolds which admit a parallel family of totally umbilical hypersurfaces.
Quart. J. Math. Oxford (2), 36 (1985), 1--15.






\bibitem{vekenetal1} Calvaruso, G., Kowalczyk, D., Van der Veken, J.:  On extrinsically symmetric hypersurfaces
of $\h^n\times\R$, Bull. Austral. Math. Soc. 82 (2010), 390--400.




\bibitem{delima-roitman} de Lima, R.F., Roitman, P.: Helicoids and catenoids in $M\times\R.$
Preprint (available at: https://arxiv.org/abs/1901.07936). 




\bibitem{lopez} López, R.: Constant mean curvature surfaces with boundary. Springer Monographs in Mathematics,  Springer (2013).



\bibitem{mendonca-tojeiro} Mendonça, B., Tojeiro, R.: Umbilical Submanifolds of $\s^n\times\R.$
Canad. J. Math. Vol. 66 (2), (2014)  400--428.


\bibitem{souam-toubiana} Souam, R., Toubiana, E.: Totally umbilic surfaces in homogeneous 3-manifolds. Comment. Math.
Helv. 84 (2009), no. 3, 673--704. 


\bibitem{souam-veken} Souam, R., Van der Veken, J.:  Totally umbilical hypersurfaces of manifolds admitting a unit Killing
field. Trans. Amer. Math. Soc. 364 (2012), no. 7, 3609--3626.


\bibitem{vekenetal} Van der Veken, J., Vrancken, L.: Parallel and semi-parallel hypersurfaces of $\s^n\times\R.$ Bull. Braz.
Math. Soc. 39(2008), no. 3, 355--370.


\end{thebibliography}
\end{document}